%% file: MarkovianRPsArxiv.tex
\documentclass[english]{amsart}

\input{preamble}

\begin{document}

\title{A support and density theorem for Markovian rough paths}

\author{Ilya Chevyrev}
\address{I. Chevyrev,
Mathematical Institute,
University of Oxford,
Andrew Wiles Building,
Radcliffe Observatory Quarter,
Woodstock Road,
Oxford OX2 6GG,
United Kingdom}
\email{chevyrev@maths.ox.ac.uk}

\author{Marcel Ogrodnik}
\address{M. Ogrodnik,
Department of Mathematics,
Imperial College London,
Huxley Building,
180 Queens Gate,
London SW7 2AZ,
United Kingdom}
\email{marcel\_ogrodnik@hotmail.com}





\begin{abstract}
We establish two results concerning a class of geometric rough paths $\mathbf{X}$ which arise as Markov processes associated to uniformly subelliptic Dirichlet forms.
The first is a support theorem for $\mathbf{X}$ in $\alpha$-H{\"o}lder rough path topology for all $\alpha \in (0,1/2)$, which proves a conjecture of Friz--Victoir~\cite{FrizVictoir10}.
The second is a H{\"o}rmander-type theorem for the existence of a density of a rough differential equation driven by $\mathbf{X}$, the proof of which is based on analysis of (non-symmetric) Dirichlet forms on manifolds.
\end{abstract}

\maketitle

\section{Introduction}

Consider a symmetric Dirichlet form on $L^2(\R^d, \lambda)$
\begin{equation}\label{eq:DirichletRd}
\EE(f,g) = \int_{\R^d} \sum_{i,j=1}^d a^{i,j}(\partial_i f) (\partial_j g)d\lambda\;,
\end{equation}
where $\lambda$ is the Lebesgue measure and $a$ is a measurable, uniformly elliptic function taking values in the space of symmetric $d\times d$ matrices (we make our set-up precise in Section~\ref{subsec:Notation}).
It is well-known that there exists a symmetric Markov process $\mathbf{X}$ in $\R^d$ associated with $\EE$;
see~\cite{Fukushima11} for a general construction of $\X$ and~\cite{Stroock88} for fundamental analytic properties of $\EE$.

We are interested in differential equations of the form
\begin{equation}\label{eq:RDEIntro}
d\Y_t = V(\Y_t)d\X_t\;, \quad \Y_0 = y_0 \in \R^e\;,
\end{equation}
driven by $\X$ along vector fields $V = (V_1,\ldots,V_d)$ on $\R^e$. When $a$ is taken sufficiently smooth, the process $\mathbf{X}$ can be realised as a semi-martingale for which the classical framework of It{\^o} gives meaning to the equation~\eqref{eq:RDEIntro}. However for irregular functions $a$, this is no longer the case, and~\eqref{eq:RDEIntro} falls outside the scope of It{\^o} calculus.

One of the applications of Lyons' theory of rough paths~\cite{Lyons98} has been to give meaning to differential equations driven by processes outside the range of semi-martingales. One viewpoint of rough paths theory is that it factors the problem of solving equations of the type~\eqref{eq:RDEIntro} into first enhancing $\X$ to a rough path by appropriately defining its iterated integrals (which is typically done through stochastic means), and then solve~\eqref{eq:RDEIntro} deterministically.

Probabilistic methods to enhance the Markov process $\X$ to a rough path and the study of its fundamental properties appear in~\cite{LyonsStoica99, BassHamblyLyons02, Lejay06, Lejay08}, where primarily the forward-backward martingale decomposition is used to show existence of the stochastic area.
A somewhat different approach, which we follow here, is taken in~\cite{FrizVictoir08} where the authors define $\X$ directly as a diffusion on the free nilpotent Lie group $G^N(\R^d)$ (in particular the iterated integrals are given directly in the construction). One can show that in the situation mentioned at the start, the two methods give rise to equivalent definitions of rough paths.
The latter construction in fact yields further flexibility in that the evolution of $\X$ can depend in a non-trivial way on its higher levels (its iterated integrals).
Note that this is a common feature with L{\'e}vy rough paths studied in~\cite{FrizShekhar17, Chevyrev18}.
Markovian rough paths have also recently been investigated in~\cite{CassOgrodnik17, ChevyrevLyons16} in connection with the accumulated local $p$-variation functional and the moment problem for expected signatures.

The goal of this paper is to contribute two results to the study of Markovian rough paths in the sense of~\cite{FrizVictoir08}. Our first contribution (Theorem~\ref{thm:support_proper}) answers in the positive a conjecture about the support of $\X$ in $\alpha$-H{\"o}lder rough path topology. Such a support theorem appeared in~\cite{FrizVictoir08} for $\alpha \in (0,1/6)$, and was improved to $\alpha \in (0,1/4)$ in~\cite{FrizVictoir10} where it was conjectured to hold for $\alpha \in (0,1/2)$ in analogy to enhanced Brownian motion. Comparing our situation to the case of Gaussian rough paths, where such support theorems are known with sharp H{\"o}lder exponents (see e.g.,~\cite[Sec.~15.8]{FrizVictoir10}, and~\cite{FrizGess16} for recent improvements), the difficulty of course lies in the lack of a Gaussian structure, in particular the absence of a Cameron-Martin space.

Our solution to this problem relies almost entirely on elementary techniques.
Indeed, we first show that any stochastic process (taking values in a Polish space) admits explicit lower bounds on the probability of keeping a small $\alpha$-H{\"o}lder norm, provided that it satisfies lower and upper bounds on certain transition probabilities comparable to Brownian motion.
This is made precise by conditions~\ref{point:supUpper} and~\ref{point:ballLower} and Theorem~\ref{thm:smallHolder}.
We then verify these conditions for the translated rough path $T_h(\X)$ (which is in general non-Markov, see Remark~\ref{remark:nonMarkov}) for any $h \in W^{1,2}([0,T],\R^d)$ using heat kernel estimates of $\X$ (we also note that, just like for enhanced Brownian motion, all relevant constants depend on $h$ only through $\norm{h}_{W^{1,2}}$).

As usual, in combination with the continuity of the It{\^o}-Lyons map from rough paths theory, an immediate consequence of improving the H{\"o}lder exponent in the support theorem for $\X$
is a stronger Stroock-Varadhan support theorem (in $\alpha$-H{\"o}lder topology) for the solution $\Y$ to the rough differential equation (RDE)~\eqref{eq:RDEIntro} along with the lower regularity assumptions on the driving vector fields $V$ ($\Lip^2$ instead of $\Lip^4$).

Our second contribution (Theorem~\ref{thm:WHormander} and its Corollary~\ref{cor:HorCond}) may be seen as a non-Gaussian H{\"o}rmander-type theorem, and provides sufficient conditions on the driving vector fields $V = (V_1,\ldots, V_d)$ under which the solution to the RDE~\eqref{eq:RDEIntro} admits a density with respect to the Lebesgue measure on $\R^e$. Once again, while this result is reminiscent of density theorems for RDEs driven by Gaussian rough paths (e.g.,~\cite{BaudoinHairer07, CassFriz10, CassHairer15}), the primary difference in our setting is that methods from Malliavin calculus are no longer available due to the lack of a Gaussian structure.

We replace the use of Malliavin calculus by direct analysis of (non-symmetric) Dirichlet forms on manifolds. Indeed, we identify conditions under which the couple $(\X,\Y)$ admits a density on its natural state-space, and conclude by projecting to $\Y$. We note however that our current result gives no quantitative information about the density beyond its existence (not even for the couple $(\X,\Y)$), and we strongly suspect that the method can be improved to yield further information (particularly $L^p$ bounds and regularity results in the spirit of the De Giorgi--Nash--Moser theorem).

\subsection{Notation}\label{subsec:Notation}

Throughout the paper, we adopt the convention that the domain of a path $\x : [0,T] \to E$, for $T > 0$ and a set $E$, is extended to all of $[0,\infty)$ by setting $\x_t = \x_T$ for all $t > T$.
For a metric space $(E,d)$, $r \geq 0$, and $x \in E$, we denote the ball $B(x,r) = \{y \in E \mid d(x,y) \leq r\}$.

We let $G = G^N(\R^d)$ denote the step-$N$ free nilpotent Lie group over $\R^d$ for some $N \geq 2$,
and let $U_1,\ldots, U_d$ be a set of generators for its Lie algebra $\g = \g^N(\R^d)$, which we identify with the space of left-invariant vector fields on $G$. We equip $\R^d$ with the inner product for which $U_1,\ldots, U_d$ form an orthonormal basis upon canonically identifying $\R^d$ with a subspace of $\g$.

We equip $G$ with the corresponding Carnot--Carath{\'e}odory metric $d$. Let $1_G$ denote the identity element of $G$ and let $\lambda$ denote the Haar measure on $G$ normalised so that $\lambda(B(1_G,1)) = 1$.

For $\Lambda > 0$, let $\Xi(\Lambda) = \Xi^{N,d}(\Lambda)$ denote the set of measurable functions $a$ on $G$ which take values in the space of symmetric $d\times d$ matrices and which are sub-elliptic in the following sense:
\[
\Lambda^{-1} |\xi|^2 \leq \gen{\xi,a(x)\xi} \leq \Lambda|\xi|^2\;, \quad \forall \xi \in \R^d\;, \quad \forall x \in G\;.
\]
For $a \in \Xi(\Lambda)$, we define the associated Dirichlet form $\EE = \EE^a$ on $L^2(G,\lambda)$ for all $f,g \in C^\infty_c(G)$ by
\begin{equation}\label{eq:DirichletG}
\EE(f,g) = \int_G \sum_{i,j} a^{i,j} (U_i f) (U_j g) d\lambda\;.
\end{equation}
We let $\X = \X^{a,x}$ denote the Markov diffusion on $G$ associated to $\EE$ with starting point $\X_0 = x \in G$. We recall that the sample paths of $\X$ are a.s. geometric $\alpha$-H{\"o}lder rough paths for all $\alpha \in (0,1/2)$, and when $a(x)$ depends only on the level-$1$ projection $\pi_1(x) \in \R^d$ of $x \in G$, $\X$ serves as the natural rough path lift of the Markov diffusion associated to the Dirichlet form~\eqref{eq:DirichletRd} on $L^2(\R^d)$ discussed earlier. For further details, we refer to~\cite{FrizVictoir10}.

\begin{remark}
Throughout the paper we assume the symmetric Dirichlet form~\eqref{eq:DirichletG} is defined on the Hilbert space $L^2(G,\lambda)$ so that $\X$ is symmetric with respect to $\lambda$. As pointed out in~\cite{CassOgrodnik17}, it is natural to also consider $\EE$ defined over $L^2(G,\mu)$ for a measure $\mu(dx) = v(x)\lambda(dx)$, $v \geq 0$. While for simplicity we only work with $\EE$ defined on $L^2(G,\lambda)$, we note that appropriate assumptions of $v$ and a Girsanov transform (see, e.g.,~\cite{Fitzsimmons97}) can be used to relate the results of this paper to this more general setting.
\end{remark}

\section{Support theorem}

\subsection{Restricted H{\"o}lder norms}

We first record some deterministic results on H{\"o}lder norms which will be used in the sequel. Throughout this section, let $(E,d)$ be a metric space, $\alpha \in (0,1]$, $T > 0$, and $\x \in C([0,T],E)$ a continuous path. Let $\square$ denote any of the relations $<, \leq, =, \geq, >$, and consider the quantity
\[
\norm{\x}_{\aHol,\square\varepsilon;[s,t]} = \sup_{u,v \in [s,t], |u-v| \square \varepsilon} \frac{d(\x_u,\x_v)}{|u-v|^\alpha}\;,
\]
where we set $\norm{\x}_{\aHol,\square\varepsilon;[s,t]} = 0$ if the set $\{(u,v) \in [s,t]^2 \mid |u-v| \square \varepsilon\}$ is empty.

\begin{definition}\label{def:tau}
For $\varepsilon, \gamma > 0$ and $s \in [0,T]$, define the times $(\tau^{\varepsilon, \gamma,s}_n)_{n \geq 0} = (\tau_n)_{n \geq 0}$ by $\tau_0 = s$ and for $n \geq 1$
\[
\tau_n = \inf\{t > \tau_{n-1} \mid \norm{\x}_{\aHol,\geq \varepsilon;[\tau_{n-1},t]} \geq \gamma\}\;.
\]
We call any such $\tau_n$ a \emph{H{\"o}lder stopping time} of $\x$.
\end{definition}

\begin{lemma}\label{lem:unifHolBound}
Let $\varepsilon,\gamma >0$ and $s = 0$, and suppose that for some $c > 0$
\begin{equation}\label{eq:supCond}
\sup_{t \in [\tau_n,\tau_n + \varepsilon]} d(\x_{\tau_n}, \x_{t}) < c\;, \quad \forall n \geq 0\;.
\end{equation}
Then $\norm{\x}_{\aHol,=\varepsilon;[0,T]} < \tilde \gamma := (3c\varepsilon^{-\alpha}) \vee (4\gamma + c\varepsilon^{-\alpha})$.
\end{lemma}

\begin{proof}
For $n \geq 1$ and $t \in [\tau_n - \varepsilon, \tau_n]$, we have one of the following three mutually exclusive cases: (a) $\tau_n = \tau_{n-1}+\varepsilon$, (b) $\tau_n \in (\tau_{n-1}+\varepsilon, \tau_{n-1}+2\varepsilon]$ and $t \in [\tau_{n-1},\tau_{n-1}+\varepsilon]$, or (c) $t >  \tau_{n-1}+\varepsilon$. In case (a),~\eqref{eq:supCond} implies that $d(\x_t,\x_{\tau_n}) < 2c$. In case (b), $d(\x_{\tau_n},\x_{\tau_{n-1}}) \leq \gamma(2\varepsilon)^\alpha$ and~\eqref{eq:supCond} implies that $d(\x_{t},\x_{\tau_{n-1}}) < c$, so that
\[
d(\x_{t},\x_{\tau_n}) < c + \gamma(2\varepsilon)^\alpha \leq (2c) \vee (4\gamma \varepsilon^{\alpha})\;.
\]
In case (c), we have $d(\x_{t-\varepsilon},\x_t) \leq \gamma\varepsilon^\alpha$ and $d(\x_{t-\varepsilon},\x_{\tau_n}) \leq \gamma(2\varepsilon)^\alpha$, so that
\[
d(\x_t,\x_{\tau_n}) \leq \gamma\varepsilon^\alpha + \gamma(2\varepsilon)^\alpha \leq 3\gamma\varepsilon^\alpha\;.
\]
Hence, in all three cases
\begin{equation}\label{eq:ttaun}
d(\x_t,\x_{\tau_n}) < (2c)\vee (4\gamma\varepsilon^{\alpha})\;.
\end{equation}
Consider now
\[
\tau = \inf\{t > 0 \mid \norm{\x}_{\aHol;=\varepsilon;[0,T]} = \tilde \gamma\}\;.
\]
Note that $\norm{\x}_{\aHol;=\varepsilon;[0,T]} \geq \tilde \gamma \Leftrightarrow \tau < \infty$. Arguing by contradiction, suppose that $\tau < \infty$, which means that $d(\x_{\tau-\varepsilon},\x_\tau) = \tilde \gamma \varepsilon^\alpha$.
Consider the largest $n$ for which $\tau_n \leq \tau$. Observe that $\tau_n \in [\tau-\varepsilon,\tau]$, since otherwise $d(\x_{\tau-\varepsilon},\x_\tau) < \gamma \varepsilon^\alpha$, which is a contradiction since $\tilde \gamma > \gamma$. It follows from~\eqref{eq:supCond} that $d(\x_{\tau_n},\x_\tau) \leq c$, and therefore by~\eqref{eq:ttaun} and the triangle inequality
\[
d(\x_{\tau-\varepsilon},\x_\tau) < c + (2c)\vee (4\gamma\varepsilon^\alpha) = \tilde \gamma\varepsilon^\alpha\;,
\]
which is again a contradiction.
\end{proof}

\begin{lemma}\label{lem:dyadics}
Suppose that $\norm{\x}_{\aHol; =2^{-n}\varepsilon;[0,T]} \leq \gamma$ for every $n > N \in \Z$. Then
\[
\norm{\x}_{\aHol; < 2^{-N}\varepsilon;[0,T]} \leq \frac{\gamma}{1-2^{-\alpha}}\;.
\]
\end{lemma}

\begin{proof}
Consider $(t-s)/\varepsilon \in (0,2^{-N})$ with binary representation $(t-s)/\varepsilon = \sum_{n = m}^\infty c_n 2^{-n}$ with $c_n \in \{0,1\}$, $m > N$, and $c_m = 1$. It follows that
\[
d(\x_s,\x_t) \leq \gamma \sum_{n=m}^\infty \varepsilon^\alpha c_n 2^{-n\alpha}\;.
\]
Since $2^{-m} \leq (t-s)/\varepsilon$, we have $\varepsilon^\alpha 2^{-n\alpha} \leq 2^{\alpha(m-n)}(t-s)^\alpha$. Hence
\[
d(\x_s,\x_t) \leq \gamma \sum_{n=m}^\infty 2^{\alpha(m-n)}(t-s)^\alpha = \frac{\gamma(t-s)^\alpha}{1-2^{-\alpha}}\;.
\]
\end{proof}

\begin{lemma}\label{lem:globalHol}
Suppose there exist $x \in E$ and $r > 0$ such that for all integers $k \geq 0$, $\x_{k\varepsilon} \in B(x,r)$ and $\norm{\x}_{\aHol; \leq \varepsilon; [k\varepsilon,(k+1)\varepsilon]} \leq \gamma$. Then
\[
\norm{\x}_{\aHol;[0,T]} \leq 2\gamma + 2r\varepsilon^{-\alpha}\;.
\]
\end{lemma}

\begin{proof}
Consider $0 \leq s < t \leq [0,T]$, and denote $s \in [k\varepsilon,(k+1)\varepsilon)$, $t \in [n\varepsilon,(n+1)\varepsilon)$. If $k=n$ there is nothing to prove, so suppose $k < n$. If $|t-s| \leq \varepsilon$, so that $n=k+1$, then
\[
d(\x_s,\x_t) \leq d(\x_s,\x_{n\varepsilon}) + d(\x_{n\varepsilon},\x_t) \leq \gamma 2^{1-\alpha}|t-s|^\alpha\;.
\]
Finally, if $|t-s| > \varepsilon$ then since $\x_{k\varepsilon},\x_{n\varepsilon} \in B(x,r)$, it follows that
\begin{align*}
|t-s|^{-\alpha}d(\x_s,\x_t)
&\leq |t-s|^{-\alpha}(d(\x_{k\varepsilon},\x_s) + d(\x_{k\varepsilon},\x_{n\varepsilon}) + d(\x_{n\varepsilon},\x_t)) \\
&\leq |t-s|^{-\alpha}(2\varepsilon^{\alpha}\gamma+ 2r)  \\
& \leq 2\gamma + 2r\varepsilon^{-\alpha}\;.
\end{align*}
\end{proof}

\subsection{Positive probability of small H{\"o}lder norm}

Suppose now $(E,d)$ is a Polish space. In this section, we give conditions under which an $E$-valued process has an explicit positive probability of keeping a small H{\"o}lder norm. We fix $\alpha \in (0,1/2)$, a terminal time $T > 0$, and an $E$-valued stochastic process $\X$ adapted to a filtration $(\FF_t)_{t \in [0,T]}$.

Consider the following conditions:
\begin{enumerate}[label={(\arabic*)}]
\item \label{point:supUpper} There exists $C_1 > 0$ such that for every $c,\varepsilon > 0$, and every H{\"o}lder stopping time $\tau$ of $\X$, a.s.
\[
\mathbb{P} \Big[\sup_{t \in [\tau,\tau+\varepsilon]} d(\X_\tau,\X_t) > c \mid \FF_\tau \Big] \leq C_1\exp\left( \frac{-c^2}{C_1\varepsilon} \right)\;.
\]

\item \label{point:ballLower} There exist $c_2, C_2 > 0$ and $x \in E$ such that for every $s \in [0,T]$ and $\varepsilon \in (0,T-s]$, a.s.
\[
\mathbb{P} \big[ \X_{s+\varepsilon} \in B(x,C_2\varepsilon^{1/2}) \mid \FF_s\big] \geq c_2 \1{\X_s \in B(x,C_2\varepsilon^{1/2})}\;.
\]
\end{enumerate}

Roughly speaking, the first condition states that the probability of large fluctuations of $\X$ over small time intervals should have the same Gaussian tails as that of a Brownian motion, while the second condition bounds from below the probability that $\X_{s+\varepsilon}$ is in a ball of radius $\sim\varepsilon^{1/2}$ given that $\X_s$ was in the same ball.

\begin{theorem}\label{thm:smallHolder}
Assume conditions~\ref{point:supUpper} and~\ref{point:ballLower}.
Fix $x$ as in~\ref{point:ballLower}.
Then there exist $C_{\ref{thm:smallHolder}}, c_{\ref{thm:smallHolder}} > 0$, depending only on $C_1,c_2, C_2, \alpha, T$, such that for every $\gamma > 0$, a.s.
\[
\PPP{\norm{\X}_{\aHol;[0,T]} < \gamma \mid \FF_0} \geq C_{\ref{thm:smallHolder}}^{-1} \exp\Big(\frac{-C_{\ref{thm:smallHolder}}}{\gamma^{2/(1-2\alpha)}}\Big)\1{\X_0 \in B(x,c_{\ref{thm:smallHolder}}\gamma^{1/(1-2\alpha)})}\;.
\]
\end{theorem}

\begin{lemma}\label{lem:lessHol}
Assume condition~\ref{point:supUpper}. Then there exists $C_{\ref{lem:lessHol}} > 0$, depending only on $C_1$ and $\alpha$, such that for all $0 \leq s < t \leq T$ and $\varepsilon \in (0,t-s]$, a.s.
\[
\PPP{\norm{\X}_{\aHol;\leq\varepsilon;[s,t]} \geq \gamma \mid \FF_s} \leq C_{\ref{lem:lessHol}} (t-s) \varepsilon^{-1}(\gamma^{-2}\varepsilon^{1-2\alpha}+1) \exp\Big(\frac{-\gamma^{2}(1-2^{-\alpha})^2}{9C_1\varepsilon^{1-2\alpha}}\Big)\;.
\]
\end{lemma}

\begin{proof}
Let $\tau_n = \tau_n^{\varepsilon,\gamma,s}$ be defined as in Definition~\ref{def:tau} with $\tau_0 = s$. Note that~\ref{point:supUpper} implies that for all $c,\gamma > 0$, $t > s$ and $\varepsilon \in (0,t-s]$,
\[
\PPP{\exists n \geq 0, \tau_n \leq t, \sup_{u \in [\tau_n,\tau_n+\varepsilon]} d(\X_{\tau_n}, \X_u) > c \mid \FF_s} \leq \roof{(t-s)/\varepsilon} C_1\exp\Big(\frac{-c^2}{C_1\varepsilon}\Big)\;,
\]
so that by Lemma~\ref{lem:unifHolBound}
\[
\PPP{\norm{\X}_{\aHol;=\varepsilon;[s,t]} \geq (3c\varepsilon^{-\alpha}) \vee (4\gamma + c\varepsilon^{-\alpha}) \mid \FF_s} \leq \roof{(t-s)/\varepsilon} C_1\exp\Big(\frac{-c^2}{C_1\varepsilon}\Big)\;.
\]
In particular, choosing $c = 2\gamma\varepsilon^{\alpha}$ yields that for all $\gamma > 0$, $t > s$, and $\varepsilon \in (0,t-s]$,
\[
\PPP{\norm{\X}_{\aHol;=\varepsilon;[s,t]} \geq 6\gamma \mid \FF_s} \leq \roof{(t-s)/\varepsilon} C_1\exp\Big(\frac{-(2\gamma)^2}{C_1\varepsilon^{1-2\alpha}}\Big)\;.
\]
Hence
\[
\PPP{\exists n \geq 0, \norm{\X}_{\aHol;=2^{-n}\varepsilon;[s,t]} \geq \gamma \mid \FF_s} \leq 2C_1(t-s)\varepsilon^{-1}\sum_{n=0}^\infty 2^n \exp\Big(\frac{-2^{n(1-2\alpha)}\gamma^2}{9C_1\varepsilon^{1-2\alpha}}\Big)\;.
\]
The conclusion now follows from Lemma~\ref{lem:dyadics} and the observation that for every $\theta > 0$ there exists $C_4$ such that for all $K > 0$
\[
\sum_{n=0}^\infty 2^n \exp\left(-K 2^{\theta n}\right) \leq C_4 (K^{-1}+1) e^{-K}
\]
(which can be seen, for example, by the integral test and the asymptotic behaviour of the incomplete gamma function $\Gamma(p,K)$).
\end{proof}

\begin{proof}[Proof of Theorem~\ref{thm:smallHolder}]
For $\gamma, \varepsilon > 0$ and $s \in [0,T]$, consider the event
\[
A_s = \{ \norm{\X}_{\aHol; \leq \varepsilon;[s,s+\varepsilon]} < \gamma, \X_{s+\varepsilon} \in B(x,C_2\varepsilon^{1/2})\}\;.
\]
Applying condition~\ref{point:ballLower} and Lemma~\ref{lem:lessHol} with $t = s+\varepsilon$, we see that for all $s \in [0,T]$, and $\varepsilon, \gamma > 0$
\[
\PPP{ A_s \mid \FF_s} \geq c_2\1{\X_s \in B(x,C_2\varepsilon^{1/2})} - C_{\ref{lem:lessHol}}(\gamma^{-2}\varepsilon^{1-2\alpha}+1) \exp\Big( \frac{-\gamma^{2}(1-2^{-\alpha})^2}{9C_1\varepsilon^{1-2\alpha}}\Big)\;.
\]
Observe also that Lemma~\ref{lem:globalHol} (with $r = C_2\varepsilon^{1/2}$) implies that for all $\varepsilon, \gamma > 0$
\[
\mathbb{P}\Big[\norm{\X}_{\aHol;[0,T]} < 2\gamma + 2C_2\varepsilon^{1/2-\alpha} \mid \FF_0\Big] \geq \mathbb{P}\Big[\bigcap_{k=0}^{\roof{T/\varepsilon}-1} A_{k\varepsilon} \mid \FF_0 \Big]\;.
\]
It remains to control the final probability on the RHS. We set $\varepsilon = c_1\gamma^{2/(1-2\alpha)}$ (so that $\varepsilon^{1/2-\alpha}\sim \gamma$), where $c_1 > 0$ is sufficiently small (and depends only on $C_1,c_2, C_2,C_{\ref{lem:lessHol}}$ and $\alpha$) such that
\[
\kappa := c_2 - C_{\ref{lem:lessHol}}(c_1^{1-2\alpha} + 1)\exp\Big(\frac{-(1-2^{-\alpha})^2}{36 C_1 c_1^{1-2\alpha}}\Big) > 0\;,
\]
so in particular for all $s \in [0,T]$ and $\gamma > 0$,
\[
\PPP{A_s \mid \FF_s} \geq \kappa \1{\X_s \in B(x,C_2\varepsilon^{1/2})}\;.
\]
Inductively applying conditional expectations, it follows that for all $n \geq 0$
\[
\mathbb{P}\Big[\bigcap_{k=0}^{n} A_{k\varepsilon} \mid \FF_0\Big] \geq \kappa^{n+1}\1{\X_0 \in B(x,C_2\varepsilon^{1/2})}\;.
\]
Taking $n = \roof{T/\varepsilon}-1$ yields the desired result.
\end{proof}

\subsection{Support theorem for Markovian rough paths}

We now turn to the support theorem for Markovian rough paths in $\alpha$-H{\"o}lder topology, which we state in Theorem~\ref{thm:support_proper} at the end of this section.

Recall the Sobolev path space $W^{1,2} = W^{1,2}([0,T], \R^d)$ and the translation operator $T_{h}(\x)$ defined for $\x \in C^{\pvar}([0,T],G)$, $1 \leq p < N+1$, and $h \in C^{\onevar}([0,T], \R^d)$ (see~\cite[Sec.~1.4.2,~9.4.6]{FrizVictoir10}).
Let us fix $\alpha \in (0,1/2)$ and $\Lambda > 0$.
Recall further Notation~\ref{subsec:Notation}, in particular the set $\Xi(\Lambda)$.

\begin{proposition}\label{prop:support}
Let $h \in W^{1,2}$.
There exists a constant $C_{\ref{prop:support}} > 0$, depending only on $\Lambda$, $\norm{h}_{W^{1,2}}$, $\alpha$, and $T$, such that for all $a \in \Xi(\Lambda)$, $x \in G$, and $\gamma > 0$
\[
\PPPover{a,x}{\norm{T_h(\X)}_{\aHol;[0,T]} < \gamma} \geq C_{\ref{prop:support}}^{-1} \exp \Big( \frac{-C_{\ref{prop:support}}}{\gamma^{2/(1-2\alpha)}}\Big)\;.
\]
\end{proposition}

For the proof, let us fix $h \in W^{1,2}$ and a filtration $(\FF_t)_{t \in [0,T]}$ to which $\X$ (and thus $T_h(\X)$) is adapted (e.g, the natural filtration generated to $\X$).

\begin{remark}\label{remark:nonMarkov}
If $a(x)$ depends only on the first level $\pi_1(x)$ for all $x \in G$, then $T_h(\X)$ is a (non-symmetric, time-inhomogeneous) Markov process.
In general, however, $T_h(\X)$ is non-Markov.
The reason is that, for any fixed $t \in (0,T]$, the sigma-algebra $\sigma(\X_t)$ is not necessarily contained in $\sigma(T_h(\X)_t)$, i.e., information on whether $T_h(\X)_t \in A$ for Borel subsets $A \subset G$ does not yield full information about $\X_t$, which is necessary to determine the evolution of $\X$, and thus of $T_h(\X)$.
\end{remark}

Recall that the Fernique estimate~\cite[Cor.~16.12]{FrizVictoir10} implies that for every stopping time $\tau$ and $p > 2$, a.s.
\begin{equation}\label{eq:Fernique}
\PPP{\norm{\X}_{\pvar;[\tau,\tau+\varepsilon]} > c \mid \FF_\tau} \leq C_F\exp\Big( \frac{-c^2}{C_F \varepsilon} \Big),
\end{equation}
where $C_F$ depends only on $\Lambda$ and $p$.
We now prove two lemmas which demonstrate that the process $T_h(\X)$ satisfies conditions~\ref{point:supUpper} and~\ref{point:ballLower}.

\begin{lemma}\label{lem:supBound}
There exists a constant $C > 0$, depending only on $\Lambda$, such that for all $c, \varepsilon > 0$ satisfying
\begin{equation}\label{eq:epsUpper}
\varepsilon \leq \frac{c^2}{4\norm{h}^2_{W^{1,2}}},
\end{equation}
it holds that for every stopping time $\tau$, a.s.
\[
\mathbb{P}\Big[\sup_{t \in [\tau,\tau+\varepsilon]} d(T_h(\X)_\tau,T_h(\X)_t) > c \mid \FF_\tau \Big] \leq C\exp\Big(\frac{-c^2}{4C\varepsilon}\Big)\;.
\]
\end{lemma}

\begin{proof}
Suppose $c,\varepsilon > 0$ satisfy~\eqref{eq:epsUpper}.
Using that $\norm{h}_{\onevar;[s,s+\varepsilon]} \leq \varepsilon^{1/2}\norm{h}_{W^{1,2};[s,s+\varepsilon]}$, we have $\norm{h}_{\onevar;[s,s+\varepsilon]} \leq c/2$.
Fix now any $2 < p < N+1$.
Observe that (see~\cite[Thm.~9.33]{FrizVictoir10})
\begin{align*}
\sup_{t \in [s,s+\varepsilon]}d(T_h(\X)_s,T_h(\X)_t) &\leq \norm{T_h(\X)}_{\pvar;[s,s+\varepsilon]} \\
&\leq C_1\left(\norm{\X}_{\pvar;[s,s+\varepsilon]} + \norm{h}_{\onevar;[s,s+\varepsilon]} \right)\;,
\end{align*}
from which the conclusion follows by the Fernique estimate~\eqref{eq:Fernique}.
\end{proof}

\begin{lemma}\label{lem:lowerBallBound}
For all $C \geq C_0(\Lambda, \norm{h}_{W^{1,2}}) > 0$, there exists $c = c(C, \Lambda, \norm{h}_{W^{1,2}}) > 0$ such that for all $x \in G$, $s \in [0,T]$, and $\varepsilon \in (0,T-s]$, a.s.
\[
\mathbb{P}\big[ T_h(\X)_{s+\varepsilon} \in  B(x, C\varepsilon^{1/2}) \mid \FF_s \big] \geq c \1{T_h(\X)_s \in B(x,C\varepsilon^{1/2})}\;.
\]
\end{lemma}

\begin{proof}
We use the shorthand notation $\Y = T_h(\X)$. For every $x,y \in G$, consider a geodesic $\gamma^{y,x} : [0,1] \to G$ with $\gamma^{y,x}_{0} = y$ and $\gamma^{y,x}_{1} = x$ parametrised at unit speed. Let $z(y,x) := \gamma^{y,x}_{1/2}$ denote its midpoint. For any $x \in G$, observe that
\begin{align*}
d(\Y_{s+\varepsilon}, x) &\leq d(\Y_{s+\varepsilon}, z(\Y_s,x)) + d(z(\Y_s,x),x) \\
&\leq d(\Y_{s,s+\varepsilon},\X_{s,s+\varepsilon}) + d(\X_{s,s+\varepsilon},\Y_s^{-1}z(\Y_s,x)) + d(z(\Y_s,x),x)\;.
\end{align*}
If $\Y_s \in B(x,r)$, then evidently $d(z(\Y_s,x),x) \leq r/2$.
Moreover, since $G$ is a homogeneous group and due to our normalisation of $\lambda$, it holds that $\lambda(B(x,r)) = r^Q$ for all $r \geq 0$ and $x \in G$, where $Q \geq 1$ is the homogeneous dimension of $G$.
Recall also the lower bound on the heat kernel~\cite[Thm.~16.11]{FrizVictoir10}
\[
p(\varepsilon,x,y) \geq C_l^{-1} \varepsilon^{-Q/2} \exp\Big(\frac{-C_l d(x,y)^2}{\varepsilon}\Big)\;, \quad \forall x,y \in G\;, \quad \forall \varepsilon > 0\;,
\]
where $C_l > 0$ depends only on $\Lambda$.
It follows that there exists $C_1 > 0$, depending only on $\Lambda$, such that, for any $r,\varepsilon > 0$ and $y \in B(1_G,r/2)$,
\begin{align*}
\PPP{d(\X_{s,s+\varepsilon}, y) < r/4}
&\geq \lambda(B(y,r/4))C_l^{-1}\varepsilon^{-Q/2}\exp\Big(\frac{-C_l r^2}{\varepsilon}\Big)
\\
&\geq \frac{C_1^{-1}r^Q}{\varepsilon^{Q/2}}\exp\Big(\frac{-C_1 r^2}{\varepsilon}\Big)\;.
\end{align*}
Note that if $\Y_s \in B(x,r)$, then necessarily $\Y_s^{-1}z(\Y_s,x) \in B(1_G,r/2)$, so we obtain for all $x \in G$, $r,\varepsilon > 0$ and $s \in [0,T]$
\[
\PPP{d(\X_{s,s+\varepsilon}, \Y_s^{-1}z(\Y_s,x)) < r/4 \mid \FF_s}
\geq \frac{C_1^{-1}r^Q}{\varepsilon^{Q/2}}\exp\Big( \frac{-C_1 r^2}{\varepsilon} \Big) \1{\Y_s \in B(x,r)}\;.
\]
Finally, by standard rough paths estimates (using that $T_h(\X)_{s,t}$ is equal to $\X_{s,t}$ plus a combination of cross-integrals of $\X$ and $h$ over $[s,t]$) we have
\begin{align*}
d(\X_{s,s+\varepsilon},\Y_{s,s+\varepsilon})
&\leq C_2\max_{i \in \{1\ldots, N\}} \Big(\sum_{k=1}^i \norm{h}_{\onevar;[s,s+\varepsilon]}^k \norm{\X}_{\pvar;[s,s+\varepsilon]}^{i-k} \Big)^{1/i} \\
&\leq C_2\max_{i \in \{1\ldots, N\}} \Big(\sum_{k=1}^i \varepsilon^{k/2}\norm{h}_{W^{1,2};[s,s+\varepsilon]}^k \norm{\X}_{\pvar;[s,s+\varepsilon]}^{i-k} \Big)^{1/i}\;.
\end{align*}
Hence, if $\norm{\X}_{\pvar;[s,s+\varepsilon]} \leq R\varepsilon^{1/2}$, then for some $C_3 > 0$ depending only on $G$
\[
d(\X_{s,s+\varepsilon},\Y_{s,s+\varepsilon}) \leq C_3\varepsilon^{1/2}(\norm{h}_{W^{1,2};[s,s+\varepsilon]}^{1/N} + \norm{h}_{W^{1,2};[s,s+\varepsilon]})(1+R^{(N-1)/N})\;.
\]
We now let $r = C\varepsilon^{1/2}$. It follows that if $C$ and $R$ satisfy
\begin{equation}\label{eq:Clower}
C \geq 4C_3(\norm{h}_{W^{1,2};[s,s+\varepsilon]}^{1/N} + \norm{h}_{W^{1,2};[s,s+\varepsilon]})(1+R^{(N-1)/N})\;,
\end{equation}
then by the Fernique estimate~\eqref{eq:Fernique}, for any $2 < p < N+1$,
\begin{align*}
\mathbb{P}\big[ d(\X_{s,s+\varepsilon},\Y_{s,s+\varepsilon}) > C\varepsilon^{1/2}/4 \mid \FF_s \big] &\leq \mathbb{P}\big[\norm{\X}_{\pvar;[s,s+\varepsilon]} > R\varepsilon^{1/2}, \mid \FF_s\big] \\
&\leq C_F \exp\Big( \frac{-R^2}{C_F} \Big)\;.
\end{align*}
It follows that if $C$ and $R$ furthermore satisfy
\begin{equation}\label{eq:clower}
c := C_1^{-1} C^{Q}\exp\left( -C_1 C^2 \right) - C_F \exp\Big( \frac{-R^2}{C_F} \Big) > 0\;,
\end{equation}
then we obtain
\[
\mathbb{P} \big[ d(\Y_{s+\varepsilon},x) < C\varepsilon^{1/2} \mid \FF_s \big] \geq c \1{\Y_s \in B(x,C\varepsilon^{1/2})}\;.
\]
We now observe that due to the factor $R^{(N-1)/N}$ in~\eqref{eq:Clower} above, there exists $C_0 > 0$, depending only on $\norm{h}_{W^{1,2}}$ and $\Lambda$, such that for every $C \geq C_0$, we can find $R > 0$ for which~\eqref{eq:Clower} and~\eqref{eq:clower} are satisfied.
\end{proof}

\begin{proof}[Proof of Proposition~\ref{prop:support}]
By Theorem~\ref{thm:smallHolder}, it suffices to check that $T_h(\X)$ satisfies conditions~\ref{point:supUpper} and~\ref{point:ballLower} with constants $C_1,c_2, C_2$ only depending on $\Lambda$ and $\norm{h}_{W^{1,2}}$. However this follows directly from Lemmas~\ref{lem:supBound} and~\ref{lem:lowerBallBound}.
\end{proof}

\begin{theorem}\label{thm:support_proper}
Let $\gamma, R > 0$.
It holds that
\begin{equation}\label{eq:dHol_lower_bound}
\inf_{x \in G} \inf_{a\in\Xi(\Lambda)} \inf_{\|h\|_{W^{1,2}} \leq R}  \PPPover{a,x}{d_{\aHol;[0,T]}(\X,S_N(h)) < \gamma} > 0\;,
\end{equation}
where $d_{\aHol;[0,T]}$ denotes the (homogeneous) $\alpha$-H{\"o}lder metric and $S_N(h)$ is the level-$N$ lift of $h$.
In particular, the support of $\X^{a,x}$ in $\alpha$-H{\"o}lder topology is precisely the closure in $C^{\Hol{\alpha}}([0,T], G)$ of $\{ x S_N(h) \mid h \in W^{1,2} \}$.
\end{theorem}

\begin{proof}
By uniform continuity of the map $(\x,h) \mapsto T_h(\x)$ on bounded sets~\cite[Cor.~9.35]{FrizVictoir10}, and the fact that $T_hT_{-h}(\x) = \x$ and $T_h(0) = S_N(h)$, there exists $\delta = \delta(\gamma,R) > 0$ such that for all $h\in W^{1,2}$ with $\|h\|_{W^{1,2}} \leq R$
\[
\|T_{-h}(\x)\|_{\Hol{\alpha};[0,T]} < \delta(\gamma,R) \Rightarrow d_{\Hol{\alpha}}(\x,S_N(h)) < \gamma\;.
\]
The bound~\eqref{eq:dHol_lower_bound} then follows from Proposition~\ref{prop:support}.
As a consequence, we see that the support of $\X^{a,x}$ contains the closure of $\{ x S_N(h) \mid h \in W^{1,2} \}$.
The reverse inclusion follows from the fact that $\X^{a,x}$ is a.s. a geometric $\alpha$-H{\"o}lder rough path, and is therefore the limit in the $d_{\Hol{\alpha};[0,T]}$ metric of lifts of smooth paths.
\end{proof}

\begin{remark}
The main difference with the approach taken in~\cite[Thm.~50]{FrizVictoir08} and~\cite[Thm~16.33]{FrizVictoir10} to prove a bound of the form $\mathbb{P}[d_{\Hol{\alpha}}(\X,S_N(h)) <\gamma] > 0$ (with $\alpha \in [0,1/6)$ and $\alpha \in [0,1/4)$ respectively) is that we do not rely on a support theorem in the uniform topology.
As a consequence, our analysis is more delicate but does not lose any power at each step, which allows us to push to the sharp H{\"o}lder exponent range $\alpha \in [0, 1/2)$.

Note also that~\cite[Thm~16.39]{FrizVictoir10} and~\cite[Cor.~46]{FrizVictoir08} give this bound for $h\equiv 0$ with the sharp range $\alpha \in [0,1/2)$.
The proof therein relies crucially on lower and upper bounds on the probability that $\X$ stays in small balls, namely $\mathbb{P}^{a,x}[\|\X\|_{0;[0,t]} < \gamma] \asymp e^{-\lambda(\gamma)t\gamma^2}$ with $0<\lambda_{\min} \leq \lambda(\gamma) \leq \lambda_{\max} < \infty$, which yields a version of Lemma~\ref{lem:lessHol} for the untranslated process $\X^{a,x}$ conditioned to stay in a small ball around $x$.
This argument is rather sensitive to the fact that for each fixed $\gamma > 0$ the same quantity $\lambda(\gamma)$ appears in the lower and upper bounds;
this is not true for the translated process $T_h(\X)$, which is the reason for our different strategy.
\end{remark}

\section{Density theorem}

\subsection{Semi-Dirichlet forms associated with H{\"o}rmander vector fields}\label{subsec:semiDir}

In this subsection, let $\OO$ be a smooth manifold and $W = (W_1,\ldots, W_d)$ a collection of smooth vector fields on $\OO$. For $z \in \OO$, let $\Lie_{z}W$ denote the subspace of $T_{z}\OO$ spanned by the vector fields $(W_1,\ldots, W_d)$ and all their commutators at $z$. We say that $W$ satisfies H{\"o}rmander's condition on $\OO$ if $\Lie_{z}W = T_{z}\OO$ for every $z \in \OO$, in which case we call $W$ a collection of H{\"o}rmander vector fields.

Fix a collection $W = (W_1,\ldots, W_d)$ of H{\"o}rmander vector fields on $\OO$ and $U \subset \OO$ an open subset with compact closure. Consider a bounded measurable function $a$ on $U$ taking values in (not necessarily symmetric) $d\times d$ matrices such that for some $\Lambda \geq 1$
\begin{equation}\label{eq:aLower}
\Lambda^{-1} |\xi|^2 \leq \gen{\xi, a(z)\xi}\;, \quad \forall \xi \in \R^d\;, \quad \forall z \in U\;.
\end{equation}
Let $\mu$ be a smooth measure on $\OO$ and define the bilinear map
\begin{align*}
\EE &: C^\infty_c(U) \times C^\infty_c(U) \to \R \\
\EE &: (f,g) \mapsto -\sum_{i,j=1}^d \int_{U} a^{i,j}(z) (W_i f)(z)(W_j^*g)(z) \mu(dz)\;,
\end{align*}
where $W_j^* = -W_j - \diver_\mu W_j$ is the formal adjoint of $W_j$ with respect to $\mu$.
In the following lemma, the $L^p$ norm $\norm{\cdot}_{p}$ for $p \in [1,\infty]$ is assumed to be on $L^p(U,\mu)$.
For background concerning (non-symmetric, semi-)Dirichlet forms, we refer to~\cite{Oshima13}.

\begin{lemma}\label{lem:L2Density}
The bilinear form $\EE$ is closable in $L^2(U, \mu)$, lower bounded, and satisfies the sector condition.
Denote by $P_t$ the associated (strongly continuous) semi-group on $L^2(U,\mu)$. Suppose further that $P_t$ is sub-Markov (so that the closed extension of $\EE$ is a lower-bounded semi-Dirichlet form) and maps $C_b(U)$ into itself. Then there exists $\nu > 2$ and  $b>0$ such that for every $x \in U$ and $t>0$ there exists $p_t(x,\cdot) \in L^2(U,\mu)$ with $\norm{p_t(x,\cdot)}_{2} \leq bt^{-\nu/2}$ such that for all $f \in L^2(U,\mu)$
\[
P_t f(x) = \int_U p_t(x,y) f(y) \mu(dy)\;.
\]
\end{lemma}

The proof of Lemma~\ref{lem:L2Density} is based on the sub-Riemannian Sobolev inequality combined with a classical argument of Nash~\cite{Nash58}.
We believe this result should be standard, but as we were unable to find a sufficiently similar form in the literature, we prefer to give a proof in Appendix~\ref{appendix:proof} (see~\cite{SCS91,Sturm95} for closely related results in the case that $\EE$ is symmetric or positive semi-definite).

Note also that in the sequel, namely in the proof of Theorem~\ref{thm:WHormander}, we will only require the fact from Lemma~\ref{lem:L2Density} that the kernel $p_t$ exists.
The bound on $\|p_t(x,\cdot)\|_{2}$ is merely a free consequence of the proof of its existence.

\subsection{Density for RDEs}

We now specialise to the setting of Markovian rough paths. Recall Notation~\ref{subsec:Notation} and consider the RDE
\begin{equation}\label{eq:RDE}
d\Y_t = V(\Y_t)d\X_t\;, \quad \Y_0 = y_0 \in \R^e\;,
\end{equation}
for smooth vector fields $V = (V_1,\ldots, V_d)$ on $\R^e$. We suppose also that $V$ are $\Lip^2$ so that~\eqref{eq:RDE} admits a unique solution. We fix also the starting point $\X_0 = x_0 \in G$ of $\X$.

For the reader's convenience, we recall the Nagano--Sussmann orbit theorem (see, e.g.,~\cite[Chpt.~5]{AgrachevSachkov04}).
\begin{theorem}[Orbit theorem, Nagano--Sussmann]\label{thm:orbit}
Let $W$ be a set of complete smooth vector fields on a smooth manifold $M$.
Let $\OO$ denote the orbit of $W$ through a point $z_0\in M$.
Then $\OO$ is a connected immersed submanifold of $M$.
Furthermore, for any $z \in \OO$,
\[
T_{z}\OO = \spn{ \mathrm{d} (P^{-1})_{P(z)} w(P(z)) \mid P \in \PP, w \in W }\;,
\]
where
\[
\PP =  \{e^{t_1 w_1}\circ\ldots \circ e^{t_k w_k} \mid t_i \in \R, \; w_i \in W , \; k \geq 1\} \subset  \Diff\, M\;.
\]
\end{theorem}

A particularly useful consequence of the orbit theorem is the following.

\begin{corollary}\label{cor:Frobenius}
Let notation be as in Theorem~\ref{thm:orbit}.
It holds that $\Lie_{z}W \subseteq T_{z}\OO$ for all $z \in \OO$.
Furthermore, $\Lie_{z}W = T_{z}\OO$ for all $z \in \OO$ if and only if $\dim \Lie_zW$ is constant in $z$.
\end{corollary}

\begin{proof}
The fact that $\Lie_{z}W \subseteq T_{z}\OO$ and the ``only if'' implication are obvious.
For the ``if'' implication, suppose $\dim \Lie_zW$ is constant in $z \in \OO$.
Then $\Lie \; W$ defines a distribution on $\OO$ (a subbundle of the tangent bundle), so the Frobenius theorem implies that $\Lie\; W$ arises from a regular foliation of $\OO$.
However, each leaf of this foliation is itself an orbit of $W$.
Therefore the foliation contains only one leaf, namely $\OO$, which concludes the proof.
\end{proof} 

Consider the manifold $G \times \R^e$. We canonically identify the tangent space $T_{(x,y)}(G \times \R^e)$ with $T_x G \oplus T_y\R^e$ and define smooth vector fields on $G \times \R^e$ by $W_i = U_i + V_i$. Let $z_0 = (x_0,y_0) \in G \times \R^e$ and denote by $\OO = \OO_{z_0}$ the orbit of $z_0$ under the collection $W = (W_1,\ldots, W_d)$.

Denote the couple $\ZZ_t = (\X_t,\Y_t)$ which is a Markov process on $G \times \R^e$. One can readily show that a.s. $\ZZ^{z_0}_t \in \OO$ for all $t > 0$ (e.g., by approximating each sample path of $\X$ in $p$-variation for some $p > 2$ by piecewise geodesic paths).

\begin{theorem}\label{thm:WHormander}
Suppose $W$ satisfies H{\"o}rmander's condition on $\OO$, i.e., $\Lie_z W = T_z\OO$ for all $z \in \OO$.
Then for all $t > 0$, $\ZZ^{z_0}_t$ admits a density with respect to any smooth measure on $\OO$.
\end{theorem}

The proof of Theorem~\ref{thm:WHormander} will be given at the end of this section.
We first state several remarks and a consequence of the theorem.

\begin{remark}\label{remark:levelOne}
Note that from Notation~\ref{subsec:Notation} we always consider $G = G^N(\R^d)$ with $N \geq 2$.
However, in the special case that $a(x)$ depends only on the first level $\pi_1(x)$ for all $x \in G^N(\R^d)$, the identical statement in Theorem~\ref{thm:WHormander} holds for the process $\ZZ_t = (\pi_1(\X_t),\Y_t) \in \R^d\times \R^e$ (the conditions change by substituting $G$ by $\R^d$ everywhere).
The reason for this is that Lemma~\ref{lem:infGen} below can be readily adjusted to give analogous infinitesimal behaviour of the process $\ZZ_t$ (now taking values in $\OO \subseteq \R^d \times \R^e$), after which the proof of the theorem carries through without change.
\end{remark}

For a statement of the density of $\Y_t$ itself, let $\OO' \subseteq \R^e$ denote the orbit of $y_0 \in \R^e$ under $V$.
\begin{lemma}\label{lem:ZImpliesY}
Suppose $\ZZ_t^{z_0}$ admits a density with respect to a smooth measure on $\OO$.
Then $\Y_t$ admits a density with respect to any smooth measure on $\OO'$.
\end{lemma}

\begin{proof}
By the description of the tangent space $T_z \OO$ in Theorem~\ref{thm:orbit}, it holds that the projection $p_2 : \OO \to \OO', (x,y) \mapsto y$, is a (surjective) submersion (in fact a smooth fibre bundle) from $\OO$ to $\OO'$.
The conclusion follows from the fact that pre-images of null-sets under submersions are null-sets for smooth measures.
\end{proof} 

Moreover, the condition in Theorem~\ref{thm:WHormander} may be restated in terms of just the driving vector fields $V = (V_1,\ldots, V_d)$ as follows.

\begin{lemma}\label{lem:equiv_Hor}
For a multi-index $I = (i_1,\ldots, i_k) \in \{1,\ldots, d\}^k$ of length $|I| = k$, denote by $V_{[I]}$ the vector field $[[\ldots[V_{i_1},V_{i_2}],\ldots],V_{i_k}]$.
It holds that $W$ satisfies H{\"o}rmander's condition on $\OO$ if and only if
\begin{equation}\label{eq:HorCond}
\textnormal{$\dim \textnormal{span}\{V_{[I]}(y) : |I| > N \} \subseteq T_y\R^e$ is constant in $y \in \OO'$.}
\end{equation}
\end{lemma}

\begin{proof}
Since the vector fields $U_1,\ldots, U_d$ are freely step-$N$ nilpotent and generate the tangent space of $G$, observe that
\begin{equation}\label{eq:dimLieW}
\dim \Lie_{(x,y)} W = \dim G + \dim \textnormal{span}\{V_{[I]}(y) : |I| > N \}\;, \quad \forall (x,y) \in G\times \R^e\;.
\end{equation}
Suppose $W$ satisfies H{\"o}rmander's condition on $\OO$. 
Then $\dim \Lie_zW$ is constant in $z \in \OO$, and by~\eqref{eq:dimLieW} it follows that~\eqref{eq:HorCond} holds.
Conversely, suppose~\eqref{eq:HorCond} holds.
It now follows from~\eqref{eq:dimLieW} that $\dim \Lie_zW$ is constant in $z \in \OO$, and thus $W$ satisfies H{\"o}rmander's condition on $\OO$ by Corollary~\ref{cor:Frobenius}.
\end{proof}

Combining Theorem~\ref{thm:WHormander} with Lemmas~\ref{lem:ZImpliesY} and~\ref{lem:equiv_Hor}, we obtain the following corollary.

\begin{corollary}\label{cor:HorCond}
Suppose condition~\eqref{eq:HorCond} holds.
Then for all $t > 0$, the RDE solution $\Y_t$ admits a density with respect to any smooth measure on $\OO'$.
\end{corollary}

\begin{remark}
Note that $\OO' = \R^e$ whenever $V$ satisfies H{\"o}rmander's condition on $\R^e$, in which case every smooth measure is equivalent to the Lebesgue measure.
\end{remark}

\begin{remark}
Following Remark~\ref{remark:levelOne}, in the case that $a(x)$ depends only on the first level $\pi_1(x)$, we are able to take $N=1$ in~\eqref{eq:HorCond} when applying Corollary~\ref{cor:HorCond}.
\end{remark}

\begin{remark}
Note that while~\eqref{eq:HorCond} (for any $N\geq 0$) implies that $V$ satisfies H{\"o}rmander's condition on $\OO'$, the reverse implication is clearly not true. In particular, we do not know if it is sufficient for $V$ to only satisfy H{\"o}rmander's condition on $\OO'$ in order for $\Y_t$ to admit a density on $\OO'$.
The difficulty of course is that unless~\eqref{eq:HorCond} is satisfied, the couple $(\X_t,\Y_t)$ will in general not admit a density in $\OO$, whereby our method of proof breaks down. 
\end{remark}

For the proof of Theorem~\ref{thm:WHormander}, we first recall for the reader's convenience the infinitesimal behaviour of the coordinate projections of $\X^a$. As before, let $\lambda$ denote the Haar measure on $G$.

\begin{lemma}\label{lem:infProj}
Let $g \in C^\infty_c(G)$. Then for all $k,l \in \{1,\ldots, d\}$
\begin{align*}
\lim_{t \rightarrow 0} t^{-1}\gen{g, \EEEover{a,\cdot}{\X_{0,t}^k}}_{L^2(G,\lambda)} &= -\sum_{j=1}^d\int_{G}a^{k,j}(x)U_j g(x)\lambda(dx)\;, \\
\lim_{t \rightarrow 0} t^{-1}\gen{g, \EEEover{a,\cdot}{\X^k_{0,t} \X^l_{0,t}}}_{L^2(G,\lambda)} &= 2\int_{G}a^{k,l}(x) g(x)\lambda(dx)\;, \\
\lim_{t \rightarrow 0} t^{-1}\gen{g, \mathbb{E}^{a,\cdot}\big[\X_{0,t}^{k,l}\big]}_{L^2(G,\lambda)} &= 0\;.
\end{align*}
\end{lemma}

\begin{proof}
This is~\cite[Lem.~27]{FrizVictoir08} extended {\it mutatis mutandis} to the general case $G^N(\R^d)$, $N \geq 1$, cf.~\cite[Prop.~16.20]{FrizVictoir10}.
\end{proof}

\begin{lemma}\label{lem:infGen}
Let $U \subset \OO$ be an open subset with compact closure. Consider the (sub-Markov) semi-group $P_t^U$ of $\ZZ_t$ killed upon exiting $U$, defined for all bounded measurable $f : U \to \R$ by
\[
P_t^U f(z) = \EEEover{z}{f(\ZZ_t)\1{\ZZ_s \in U, \forall s \in [0,t]}}\;.
\]
Then $P^U_t$ maps $C_b(U)$ into itself, and for any smooth measure $\mu$ on $\OO$ it holds that for all $f,g \in C^\infty_c(U)$
\begin{equation}\label{eq:EEU}
\lim_{t\rightarrow 0} t^{-1}\gen{P^U_t f - f, g}_{L^2(U,\mu)} = \sum_{i,j=1}^d \int_{U} a^{i,j}(p_1(z)) (W_i f)(z)(W_j^*g)(z) \mu(dz)\;,
\end{equation}
where $p_1 : \OO \to G$ is the projection $(x,y) \mapsto x$ and $W_j^* = -W_j - \diver_\mu(W_j)$ is the adjoint of $W_j$ in $L^2(U,\mu)$.
\end{lemma}

\begin{proof}
To show that $P_t^U$ maps $C_b(U)$ into itself, let $f \in C_b(U)$. As $z_n = (x_n,y_n) \rightarrow z = (x,y)$ in $U$, it holds in particular that $x_n \rightarrow x$ in $G$. It follows that $\X^{a,x_n} \convd \X^{a,x}$ in $\alpha$-H{\"o}lder topology for any $\alpha \in [0,1/2)$~\cite[Thm.~16.28]{FrizVictoir10}, and we readily obtain that $P_t^U f(z_n) \rightarrow P_t^U f(z)$. Hence $P_t^U f \in C_b(U)$, so indeed $P_t^U$ maps $C_b(U)$ into itself.

It remains to verify~\eqref{eq:EEU}. Note that for every $z \in U$ the probability that $\ZZ^{z}$ leaves $U$ in $[0,t]$ is bounded above by $C^{-1}\exp(-Ct^{-1})$ for some $C = C(z,U,\Lambda) > 0$ (see, e.g., the Fernique estimate~\eqref{eq:Fernique}). It follows by a localisation argument and the stochastic Taylor expansion (e.g.,~\cite[Lem.~26]{FrizVictoir08}), that
\begin{align}\label{eq:inf1}
\nonumber \lim_{t\rightarrow 0} t^{-1}\gen{P^U_t f - f, g}_{L^2(U,\mu)}
=& \lim_{t \rightarrow 0}t^{-1} \int_U \Big(\sum_{i=1}^d W_i f(z) \EEEover{x}{\X_{0,t}^{ i}} \\
&+ \sum_{i,j=1}^d\frac{1}{2}W_iW_jf(z) \mathbb{E}^{x}\big[\X_{0,t}^{i}\X_{0,t}^{x;j}\big] \\
\nonumber &+ \sum_{i,j=1}^d\frac{1}{2}[W_i,W_j]f(z)\mathbb{E}^{x}\big[\X_{0,t}^{i,j}\big] \Big) g(z)\mu(dz)\;.
\end{align}
Since $p_1 : \OO \to G$ is a (surjective) submersion (in fact a smooth fibre bundle), by integrating over the fibres (e.g.,~\cite[p.~307]{GuilleminSternberg77}) we can associate to any $v \in C^\infty_c(U)$ a function $\hat v \in C^\infty_c(G)$ such that for any bounded measurable $	h : G \to \R$ 
\[
\int_U (h \circ p_1)(z) v (z) \mu(dz) = \int_{G} h(x) \hat v(x)\lambda(dx)\;.
\]
In particular, setting $v_i := (W_if)g$, $v_{i,j} := (W_iW_jf)g$ and $w_{i,j} := ([W_i,W_j]f)g$, we can apply Lemma~\ref{lem:infProj} to obtain that~\eqref{eq:inf1} equals
\begin{equation}\label{eq:inf2}
\sum_{i,j=1}^d \int_{G} \left[-a^{i,j}(x)(U_j \hat v_i)(x) + a^{i,j}(x)\hat v_{i,j}(x) \right] \lambda(dx)\;.
\end{equation}
It remains to show that~\eqref{eq:inf2} agrees with the RHS of~\eqref{eq:EEU}. To this end, we may assume by a limiting procedure that $a$ is smooth, and note that the same argument as in~\cite[p.~503]{FrizVictoir08} applies {\it mutatis mutandis} to our current setting.
\end{proof}

\begin{proof}[Proof of Theorem~\ref{thm:WHormander}]
Consider an increasing sequence of relatively compact open sets $(U_n)_{n \geq 1}$ such that $\cup_{n \geq 1} U_n = \OO$. By Lemma~\ref{lem:infGen}, we can apply Lemma~\ref{lem:L2Density} to conclude that for every $x \in \OO$ and $n \geq 1$ such that $x \in U_n$, there exists a non-negative kernel $p^n_t(x,\cdot) \in L^2(U_n,\mu)$ such that $P^{U_n}_t f(x) = \gen{p^n_t(x,\cdot), f}_{L^2(U_n,\mu)}$ for all $f \in C_b(U_n)$. 
Moreover, by definition of $P^{U_n}_t$, the sequence $p^n_t(x,\cdot)$ is increasing in $n$ and satisfies $\norm{p^n_t(x,\cdot)}_{L^1(U_n,\mu)} \leq 1$. Hence the limit $p_t(x,\cdot) := \lim_{n \rightarrow \infty}p^n_t(x,\cdot)$ is almost everywhere finite and gives precisely the transition kernel of the Markov process $\ZZ_t$ in $\OO$ with respect to $\mu$.
\end{proof}

\begin{remark}\label{remark:precompacts}
The pre-compact subsets $U_n$ were considered in the proof only to obtain existence of $p^n_t$ from Lemma~\ref{lem:L2Density} for each $n \geq 1$.
We could have avoided considering such a compact exhaustion by formulating Lemma~\ref{lem:L2Density} without a pre-compactness assumption on $U$
(however, at least without extra assumptions, the proof of such a formulation itself would seem to require a compact exhaustion).
\end{remark}

\appendix

\section{Proof of Lemma~\ref{lem:L2Density}} \label{appendix:proof}

We follow the notation from Section~\ref{subsec:semiDir}. For $f \in C^\infty_c(U)$ denote
\[
\norm{Wf}^2_2 := \sum_{i=1}^d\norm{ W_i f}_2^2\;,
\]
and for $\alpha > 0$
\[
\EE_\alpha(f,f) := \EE(f,f) + \alpha\norm{f}_2^2\;.
\]

\begin{lemma}
\begin{enumerate}
\item For every $\varepsilon < \Lambda^{-1}$, there exists $\alpha > 0$, depending only on $\varepsilon$, $\Lambda$, $\norm{a}_\infty$ and $\sum_{i=1}^d\norm{\diver_\mu W_i}_\infty$, such that for all $f,g \in C^\infty_c(U)$
\begin{equation}\label{eq:lowerBounded}
\EE_\alpha(f,f) \geq \varepsilon \norm{W f}_2^2\;.
\end{equation}

\item There exist $\beta > 0$, depending only on $\norm{a}_\infty$ and $\sum_{i=1}^d\norm{\diver_\mu W_i}_\infty$, such that
\begin{equation}\label{eq:equivNorms}
|\EE(f,g)| \leq \beta \norm{Wf}_2\left(\norm{Wg}_2 + \norm{g}_2\right)\;.
\end{equation}
\end{enumerate} 
\end{lemma}

\begin{proof}
By the Cauchy-Schwartz inequality and~\eqref{eq:aLower}, for some $C_1, \alpha > 0$
\begin{align*}
\EE(f,f) &= \sum_{i,j=1}^d \int_{U}a^{i,j}(z)W_i f(z) W_j f(z)\mu(dz) + \sum_{i,j=1}^d \int_U a^{i,j}(z)W_i f(z) \diver_\mu W_j(z) f(z)\mu(dz) \\
&\geq \sum_{i=1}^d \int_{U} \Lambda^{-1} |W_i f(z)|^2 \mu(dz) - \sum_{i,j=1}^d\norm{\diver_\mu W_j}_\infty\norm{a^{i,j}}_\infty \norm{W_i f}_2\norm{f}_2 \\
&\geq \Lambda^{-1}\norm{W f}_2^2  - C_1\norm{W f}_2\norm{f}_2 \\
&\geq \varepsilon\norm{Wf}_2^2 - \alpha\norm{f}_2^2\;,
\end{align*}
which implies~\eqref{eq:lowerBounded}. On the other hand, by Cauchy-Schwartz, for some $C_2,C_3 > 0$
\[
\Big|\sum_{i,j=1}^d \int_U a^{i,j}(z)W_i f(z) W_j g(z)dz \Big| \leq C_2 \norm{Wf}_2\norm{Wg}_2\;,
\]
and
\[
\Big|\sum_{i,j=1}^d\int_U a^{i,j}(z) W_i f(z) \diver_\mu W_j(z) g(z)dz \Big| \leq C_3\norm{Wf}_2 \norm{g}_2\;,
\]
from which we obtain~\eqref{eq:equivNorms}.
\end{proof}

Since $W$ satisfies H{\"o}rmander's condition on $\OO$, recall that for every $x \in \OO$ there exist constants $\nu_x > 2$, $C_x > 0$, and a neighbourhood $U_x$ of $x$ with $\mu(U_x) < \infty$ such that for all $f \in C^\infty_c(U_x)$ (see, e.g.,~\cite[p.~296]{Sturm95})
\[
\Big(\int_{U_x} |f|^{2\nu_x/(\nu_x-2)} d\mu\Big)^{(\nu_x-2)/\nu_x} \leq C_x \int_{U_x} \Big(\sum_{i=1}^d |W_i f|^2 + |f|^2 \Big) d\mu\;.
\]
Since $U$ is pre-compact, it is routine to patch together such inequalities using a partition of unity and apply interpolation to arrive at the following Sobolev inequality.

\begin{lemma}[Sobolev inequality]\label{lem:Sobolev}
There exist constants $\nu > 2$ and $C, R > 0$ such that for all $f\in C^\infty_c(U)$ with $\norm{f}_1 \leq 1$ and $\norm{f}_2 > R$
\[
\norm{f}_{2\nu/(\nu-2)}^{2} \leq C \norm{Wf}_2^2\;.
\]
\end{lemma}

Fix $\varepsilon < \Lambda^{-1}$ and $\alpha > 0$ such that~\eqref{eq:lowerBounded} holds. Let $\FF$ be the closure of $C^\infty_c(U)$ under $\norm{\cdot}_\FF := \EE_\alpha(\cdot,\cdot)^{1/2}$.

\begin{corollary}[Nash inequality]\label{cor:EELowerBound}
Let $\nu > 2$ and $R>0$ be the same as in Lemma~\ref{lem:Sobolev}. There exists $c >0$ such that for all $f \in \FF$ with $\norm{f}_1 \leq 1$ and $\norm{f}_2 > R$, it holds that
\begin{equation}\label{eq:Nash}
\EE(f,f) \geq c\norm{f}^{2+4/\nu}_2\;.
\end{equation}
\end{corollary}

\begin{proof}
Consider first $f \in C^\infty_c(U)$. The Sobolev inequality (Lemma~\ref{lem:Sobolev}), along with H{\"o}lder's inequality, implies that
\[
\norm{f}_2^{2+4/\nu} \leq C\norm{f}_1^{4/\nu}\norm{W f}^2_2\;,
\]
from which the conclusion follows first for all $f \in C^\infty_c(U)$ by~\eqref{eq:lowerBounded}, and then for general $f \in \FF$ by an approximation argument.
\end{proof}

\begin{proof}[Proof of Lemma~\ref{lem:L2Density}]
The desired properties of $\EE$ all follow from~\eqref{eq:lowerBounded} and~\eqref{eq:equivNorms} and the fact that each $W_i$ is a closable operator defined on $C^\infty_c(U) \subset L^2(U,\mu)$.

Denote by $A$ the generator of the associated adjoint semi-group $P_t^*$ in $L^2(U,\mu)$ with domain $D(A)$.
Consider $f \in D(A)$ with $\norm{f}_1 \leq 1$ and set $u_t = P_t^* f$. Since $P_t$ is sub-Markov, we have $\norm{u_t}_1 \leq 1$, so by Corollary~\ref{cor:EELowerBound}, whenever $\norm{u_t}_2 > R$,
\[
\frac{d}{dt}\norm{u_t}_2^2 = \lim_{h \rightarrow 0} \frac{\norm{P_h^* u_t}^2_2-\norm{u_t}^2_2}{h} = -2\EE(u_t,u_t) \leq -2c\norm{u_t}_2^{2+4/\nu}\;,
\]
from which it follows that there exists $b>0$ such that $\norm{P_t^* f}_2 \leq b t^{-\nu/2}$.

To complete the proof, it remains only to apply an approximation of the Dirac delta $\gen{\phi,f_n} \rightarrow \gen{\phi,\delta_x} = \phi(x)$ for all $\phi \in C_b(U)$, with $f_n \in D(A)$ and $\norm{f_1} \leq 1$, and use the fact that $\sup_n \norm{P_t^* f_n}_2 \leq b t^{-\nu/2}$ and that $P_t$ preserves $C_b(U)$.
\end{proof}

\bibliographystyle{./Martin}
\bibliography{./AllRefs}

\end{document}

%% file: preamble.tex
\usepackage{graphicx}
\usepackage{amsmath,amsfonts,amsthm}
\usepackage{amssymb}
\usepackage{multirow,array}
\usepackage{comment}
\usepackage{enumitem}
\usepackage[utf8]{inputenc}  
\usepackage[T1]{fontenc}
\usepackage{emptypage}
\usepackage{hyperref}

\numberwithin{equation}{section}
\numberwithin{figure}{section}
\newcounter{dummy}
\numberwithin{dummy}{section}

  \theoremstyle{plain}
  \newtheorem*{lemma*}{Lemma}
  \theoremstyle{plain}
  \newtheorem*{corollary*}{Corollary}
  \theoremstyle{plain}
  \newtheorem{lemma}[dummy]{Lemma}
  \theoremstyle{plain}
\newtheorem{corollary}[dummy]{Corollary}
  \theoremstyle{plain}
  \newtheorem{theorem}[dummy]{Theorem}
  \newtheorem{proposition}[dummy]{Proposition}
  \theoremstyle{remark}
  \newtheorem{remark}[dummy]{Remark}
  \theoremstyle{definition}

\newtheorem{definition}[dummy]{Definition}

\newcommand{\Lip}{\textnormal{Lip}}

\newcommand{\Lie}{\textnormal{Lie}}

\newcommand{\diver}{\textnormal{div}}

\newcommand{\OO}{\mathcal{O}}

\newcommand{\PP}{\mathcal{P}}
\newcommand{\DD}{\mathcal{D}}
\newcommand{\EE}{\mathcal{E}}

\newcommand{\FF}{\mathcal{F}}

\newcommand{\R}{\mathbb R}

\newcommand{\Z}{\mathbb Z}

\newcommand{\X}{\mathbf X}
\newcommand{\Y}{\mathbf Y}
\newcommand{\x}{\mathbf x}

\newcommand{\ZZ}{\mathbf Z}

\newcommand{\Hol}[1]{#1\textnormal{-H{\"o}l}}
\newcommand{\pvar}{p\textnormal{-var}}

\newcommand{\onevar}{1\textnormal{-var}}

\newcommand{\aHol}{\alpha\textnormal{-H{\"o}l}}

\newcommand{\EEEover}[2]{\mathbb E^{#1} \left[ #2 \right]}
\newcommand{\PPPover}[2]{\mathbb P^{#1} \left[ #2 \right]}

\newcommand{\Diff}{\textnormal{Diff}}

\newcommand{\PPP}[1]{\mathbb{P} \left[ #1 \right]}

\newcommand{\spn}[1]{\textnormal{span}\left\{ #1 \right\}}

\newcommand{\norm}[1]{\|#1 \|}

\newcommand{\1}[1]{\mathbf 1 \{#1\}}

\newcommand{\g}{\mathfrak g}

\newcommand{\gen}[1]{\langle #1 \rangle}

\newcommand{\roof}[1]{\lceil #1 \rceil}

\newcommand\restr[2]{{\left.\kern-\nulldelimiterspace 
  #1 
  \vphantom{\big|} 
  \right|_{#2} 
  }}

\def\convd{\,{\buildrel \DD \over \rightarrow}\,}

\makeatother

\usepackage[english]{babel}